\documentclass[a4paper,10pt]{amsart}
\usepackage[centertags]{amsmath}
\usepackage[english]{babel}
\usepackage{amsfonts}
\usepackage{amssymb}
\usepackage{amsthm}
\usepackage[usenames,dvipsnames]{color}
\usepackage{newlfont}
\usepackage{graphicx}
\usepackage{epstopdf}
\usepackage{cite}
\usepackage{bm}
\usepackage{blkarray}
\usepackage{multirow}
\usepackage{multicol}
\usepackage[all,cmtip]{xy}
\usepackage{xypic}
\usepackage{longtable}

\newcommand{\xdownarrow}[1]{%
    {\left\downarrow\vbox to #1{}\right.\kern-\nulldelimiterspace}}

\newcommand{\Z}{\mathbb{Z}}

\newcommand{\Q}{\mathbb{Q}}

\newcommand{\bq }{\begin{equation}}
\newcommand{\eq }{\end{equation}}
\theoremstyle{plain}
\newtheorem{thm}{Theorem}[section]
\newtheorem{lem}[thm]{Lemma}

\newtheorem{prop}[thm]{Proposition}

\newtheorem{cor}[thm]{Corollary}
\newtheorem{rem}[thm]{Remark}

\theoremstyle{definition}
\newtheorem{defn}[thm]{Definition}

\theoremstyle{example}

\title{A note on the stratification by automorphisms of smooth plane curves of genus $6$}

\author[E. Badr] {Eslam Badr}
\address{$\bullet$\,\,Eslam Essam Ebrahim Farag Badr}
\address{Departament Matem\`atiques, Edif. C, Universitat Aut\`onoma de Barcelona\\
08193 Bellaterra, Catalonia, Spain} \email{eslam@mat.uab.cat}
\address{Department of Mathematics,
Faculty of Science, Cairo University, Giza-Egypt}
\email{eslam@sci.cu.edu.eg}

\author[E. Lorenzo]{Elisa Lorenzo Garc\'ia}
\address{$\bullet$\,\,Elisa Lorenzo Garc\'ia}
\address{IRMAR - Universit\'e de Rennes 1\\
35042 Rennes Cedex, France}
\email{elisa.lorenzogarcia@univ-rennes1.fr}

\thanks{E. Badr is supported by MTM2013-40680-P}


\begin{document}

\maketitle

\begin{abstract} In this note, we give a so-called representative classification for the strata by automorphism group of smooth $\overline{k}$-plane curves of genus $6$, where $\overline{k}$ is a fixed separable closure of a field $k$ of characteristic $p=0$ or $p>13$. We start with a classification already obtained in \cite{BaBa3} and we mimic the techniques in \cite{LeRiRo} and \cite{Loth}.
	
Interestingly, in the way to get these families for the different strata, we find two remarkable phenomenons that did not appear before. One is the existence of a non $0$-dimensional \textit{final stratum} of plane curves. At a first sight it may sound odd, but we will see that this is a normal situation for higher degrees and we will give a explanation for it.

We explicitly describe representative families for all strata, except for the stratum with automorphism group $\Z/5\Z$. Here we find the second difference with the lower genus cases where the previous techniques do not fully work. Fortunately, we are still able to prove the existence of such family by applying a version of L\"{u}roth's theorem in dimension $2$.
\end{abstract}

\section{Introduction}
Let $k$ be a field of characteristic $p\geq0$, and fix a separable algebraic closure $\overline{k}$ of $k$. By a smooth $\overline{k}$-plane curve $C$ of genus $g$, we mean a smooth projective curve $C$ that admits a non-singular plane model $\{F_C=0\}\subseteq\mathbb{P}^{2}_{\overline{k}}$ over $\overline{k}$ of degree $d$, in this case, the genus $g$ equals $\frac{1}{2}(d-1)(d-2)$.

It might happen for some genus $g$ that no smooth $\overline{k}$-plane curves exist. The first genus for which there exist smooth $\overline{k}$-plane curves are: $0,1,3,6,...$ The curves of genus $0$ are isomorphic to the projective line, and the curves of genus $1$ are elliptic ones, which are quite well understood. For genus $3$, we get plane quartic curves, and different arithmetics properties have been investigated by many people around. We mention, for example, a classification up to isomorphism with good properties that can be found in \cite{LeRiRo,Loth}, or the study of their twists in \cite{Loth,Lo2}.

For genus $6$, the dimension of the (coarse) moduli space $\mathcal{M}_6$ of smooth curves of genus $g=6$ over $\overline{k}$ is equal to $3g-3=15$. The stratum $\mathcal{M}_6^{Pl}$ of smooth $\overline{k}$-plane curves of genus $6$ has dimension equal to $21-9=12$, since there are $21$ monic monomial of degree $5$ in $3$ variables and all the isomorphisms are given by projective matrices of size $3\times 3$. In particular, this dimension is larger than the dimension of the hyperelliptic locus, which is $2g-1=11$. A classification, up to $\overline{k}$-isomorphism, to smooth plane curves of genus $6$, and the determination of the full automorphism groups are given by the first author, et al. in \cite{BaBa3}.

The aim of this note is to refine the classification made in \cite{BaBa3}, by providing better families for the strata in $\mathcal{M}_6^{Pl}$. By better families, we mean that the set of normal forms with parameters describing the different strata of $\mathcal{M}_6^{Pl}$ give a bijection with the points in each stratum and the parameters are defined over the field of moduli of the representing point.  In particular, we look for complete and representative families over $k$, which are good substitutes to the notion of \emph{universal families} of (coarse) moduli spaces, especially when the spaces have no extra structures. These concepts have been introduced by R. Lercier, et al. in \cite[\S2]{LeRiRo}.

We find a representative family for each of the strata in \cite{BaBa3} by using the techniques in \cite{LeRiRo, Loth} except for a special one. For this family, we do not write down the representative family, but we prove that it exists and we give a recipe to construct it. We do it by applying  a version of L\"{u}roth's theorem in dimension $2$.

Another interesting phenomenon appearing is the existence of a final stratum of plane curves whose dimension is not zero. By a final stratum we mean a stratum not containing any other proper stratum. One could expect that by adding restrictions in the parameters of a family defining a stratum with a given automorphisms group, one could get bigger automorphism groups until obtaining a zero-dimensional stratum. This happens for all the families except for one. For this family each restriction in the parameters providing a bigger automorphism group yields a singular curve. We find an explanation for this fact: this family can be embedded in a family of curves of genus $6$ with the same automorphism group for which we can carry out the previous operation without getting singular curves, the key point is they are not plane curves anymore. Moreover, we prove that this may happen in general for higher genera.

Finally, we mention that the obtained representative classification for smooth quintic curves by automorphism groups is the first step for the computation of the twists of these curves. This will be done by the first author in his thesis \cite{Es}.

\subsection*{Acknowledgments}
The authors would like to thank Francesc Bars Cortina for his fruitful comments and suggestions for the preparation of this note.

\section{Preliminaries}

\subsection{Complete and representative families}
It is well known that the (coarse) moduli spaces $\mathcal{M}_g$ are algebraic varieties whose geometric points give a classification of isomorphism classes of smooth curves of genus $g$. The existence of universal families for a moduli space helps to recover the information on its points and allows to write down the attached objects to a point of this space. However, universal families do not exist for the moduli space $\mathcal{M}_g$.

R. Lercier, et al in \cite[\S2]{LeRiRo} introduced three good substitutes for the notion of universal family in our case: complete, finite and representative families.

We recall the basic definitions of such families over a field $k$ of characteristic $p=0$ or $p>2g+1$. For more details, we refer to \cite[\S2]{LeRiRo}:

\begin{defn}
Let $\verb"S"$ be a scheme over $k$. A curve of genus $g\geq2$ over $\verb"S"$ is a morphism of schemes $\verb"C"\rightarrow\verb"S"$ that is proper and smooth with geometrically irreducible fibers of dimension $1$ and genus $g$.
\end{defn}

\begin{defn}\label{FamComRep}[complete, finite, representative family]
Let $\verb"C"$ be a curve over a scheme $\verb"S"$, and assume that each geometric fiber of $\verb"C"$ corresponds to a point of a fixed stratum $S\subseteq M_g$. We then get a morphism $f_{\verb"C"}: \verb"S"\rightarrow S$ over $k$. The family $\verb"C"\rightarrow\verb"S"$ is \emph{complete} (resp. \textit{representative}) for the stratum $S$ if $f_{\verb"C"}$ is surjective (resp. bijective) on $F$-points for every algebraic extension $F/k$. If the family is complete and all the fibers of $f_{\verb"C"}$ are finite and with bounded cardinality, we say that the family is \textit{finite}. In particular, if a family is finite the dimension of the family is equal to the dimension of the scheme $\verb"S"$.

The family $\verb"C"\rightarrow\verb"S"$ is \textit{geometrically complete} (\textit{representative}) if it is complete (\textit{representative}) after extending the scalers to $\overline{k}$.
\end{defn}

We also remark the following:
\begin{rem}\label{defmod} If a family defined over $k$ is geometrically representative, then it is representative and complete (see \cite[Lemma 2.2]{LeRiRo}). Moreover, in this case the field of moduli is a field of definition.
\end{rem}

\subsection{Smooth plane curves of genus $6$: the strata $\widetilde{\mathcal{M}_6^{Pl}}(G)$}
Consider the set $\mathcal{M}_g^{Pl}$ of all isomorphism classes of smooth curves of $\mathcal{M}_g$, admitting a non-singular plane model over $\overline{k}$. For simplicity, we call such curves smooth $\overline{k}$-plane curves. For a finite group $G$, we denote by $\mathcal{M}_g^{Pl}(G)$ the set of all elements $[C]$ of $\mathcal{M}_g^{Pl}$ such that $G$ is isomorphic to a subgroup of the full automorphism group $Aut(C)$. We use the notation $\widetilde{\mathcal{M}_g^{Pl}}(G)$ for the subset of $\mathcal{M}_g^{Pl}(G)$ whose elements satisfies $G\cong Aut(C)$.

The finite groups $G$ for which $\widetilde{\mathcal{M}_6^{Pl}}(G)$ is non-empty are determined in \cite{BaBa3}. Moreover, geometrically complete families that depend on a fixed injective representation $\varrho:\,\operatorname{Aut}(C)\hookrightarrow\operatorname{PGL}_{3}(\overline{k})$ are given.
\begin{thm}[Badr-Bars]\label{autoGps}
Let $\overline{k}$ be a fixed separable closure of a field $k$ of characteristic $p$ with $p=0$ or $p>13$.
The following table gives the complete list of automorphism
groups of non-singular plane curves of degree $5$ over $\overline{k}$, along with geometrically complete families over $k$ for the associated strata. We denote by $L_{i,B}$ a homogeneous polynomial of degree $i$ in the variables $\{X,Y,Z\}\setminus \{B\}$.
\scriptsize
\begin{center}
\begin{longtable}{|c|c|c|c|}
  \hline
  Case & $G$ &$\varrho(G)$& $F_{\varrho(G)}(X,Y,Z)$  \\
  \hline
  1&$\operatorname{\operatorname{GAP}}(150,5)$\footnote{We use the $\operatorname{GAP}$ library \cite{GAP} notations for small finite group.}& $[\zeta_{5}X:Y:Z],[X:\zeta_{5}Y:Z]$ & $X^5+Y^5+Z^5$  \\
           && $[X:Z:Y],\,[Y:Z:X]$   & \\ \hline
  2&$\operatorname{GAP}(39,1)$& $[X:\zeta_{13} Y:\zeta_{13}^{10} Z],[Y:Z:X]$& $X^4Y+Y^4Z+Z^4X$   \\ \hline
  3&$\operatorname{GAP}(30,1)$ & $[X:\zeta_{15} Y:\zeta_{15}^{11}Z],[X:Z:Y]$ & $X^5+Y^4Z+YZ^4$\\ \hline
    4&$\Z/{20}\Z$& $[X:\zeta_{20}^4 Y:\zeta_{20}^5 Z]$ & $X^5+Y^5+XZ^4$    \\ \hline
  5&$\Z/{16}\Z$& $[X:\zeta_{16} Y:\zeta_{16}^{12} Z]$& $X^5+Y^4Z+XZ^4$    \\ \hline
6&$\Z/{10}\Z$& $[X:\zeta_{10}^2Y:\zeta_{10}^5Z]$& $X^5+Y^5+ XZ^4+\beta_{2,0}X^3Z^2$\\ \hline
7&$D_{10}$& $[X:\zeta_5 Y:\zeta_5^2 Z],\, [X:Z:Y]$ & $X^5+Y^5+Z^5+\beta_{3,1}XY^2Z^2+\beta_{4,3}X^3YZ$\\ \hline
8&$\Z/8\Z$&$[X:\zeta_{8} Y:\zeta_{8}^{4} Z]$& $X^5+Y^4Z+XZ^4+\beta_{2,0}X^3Z^2$
\\ \hline
 9&$S_3$& $[X:\zeta_3Y:\zeta_3^2Z]$& $X^5+Y^4Z+YZ^4+\beta_{2,1}X^3YZ+\beta_{3,3} X^2\big(Z^3+Y^3\big)+$  \\
 && $[X:Z:Y]$  &$+\beta_{4,2}XY^2Z^2$ \\ \hline
10&$\Z/5\Z$& $[X:Y:\zeta_5Z]$& $Z^5+L_{5,Z}$\\ \hline
11&$\Z/4\Z$&  $[X:\zeta_4Y:\zeta_4^2Z]$& $X^5+X\big(Z^4+Y^4\big)+\beta_{2,0}X^3Z^2+\beta_{3,2}X^2Y^2Z+\beta_{5,2}Y^2Z^3$   \\ \hline
 12&$\Z/4\Z$& $[X:Y:\zeta_4Z]$& $Z^4L_{1,Z}+L_{5,Z}$     \\ \hline
 13&$\Z/3\Z$& $[X:\zeta_3Y:\zeta_3^2Z]$& $X^5+Y^4Z+YZ^4+\beta_{2,1}X^3YZ+$  \\
&&&
$+X^2\big(\beta_{3,0}Z^3+\beta_{3,3}Y^3\big)+\beta_{4,2}XY^2Z^2$  \\ \hline
 14&$\Z/2\Z$&  $[X:Y:\zeta_2Z]$& $Z^4L_{1,Z}+Z^2L_{3,Z}+L_{5,Z}$     \\ \hline
15&$\{1\}$&  $[X:Y:Z]$& $L_5(X,Y,Z)$     \\ \hline
\caption{Geometrically complete families over $k$}\label{table:FullAuto.}
\end{longtable}
\end{center}
\normalsize
\vspace{-1cm}

The algebraic restrictions on the coefficients, so that each family is smooth, geometrically irreducible, and has no larger automorphism group are not given.
\end{thm}

\begin{rem}
The geometrically complete families over $k$ given in Theorem \ref{autoGps} are not necessarily representative or even complete.
For example, the smooth $\overline{\Q}$-plane curve $C:\,X^5+Y^5+\frac{1}{2}XZ^4+X^3Z^2=0$ has two representatives in the above list, which are $X^5+Y^5+XZ^4\pm\sqrt{2}X^3Z^2=0$ and none of them is defined over $\Q$.
\end{rem}

\section{Representative families for $\widetilde{\mathcal{M}_6^{Pl}}(\varrho(G))$}
This section is devoted to give representative families over $k$ for the different strata in Theorem \ref{autoGps} for $\overline{k}$-plane curves of genus $6$. As in \cite{LeRiRo}, the first step is computing the normalizers of the groups $\rho(G)$ in Theorem \ref{autoGps}.

Throughout this section, $k$ has characteristic $p=0$ or $p>13$.
\subsection{Computation of the normalizers $N_{\varrho(G)}(\overline{k})$}
We start with the families in Theorem \ref{autoGps}, which are geometrically complete over $k$ for each of the strata $\widetilde{\mathcal{M}_6^{Pl}}(\varrho(G))$. Isomorphisms between two curves in the same family, in particular with identical automorphism group $\varrho(G)$ in $\text{PGL}_3(\overline{k})$, are clearly given by $3\times3$ projective matrices in the normalizer $N_{\varrho(G)}(\overline{k})$ over $\overline{k}$.

\textbf{Notations.}
By $D(\overline{k})$, we mean the group of diagonal matrices in $\operatorname{PGL}_3(\overline{k})$, and by $T_X(\overline{k})$, we denote its subgroup of elements the shapes $[a X:Y:Z]$ for some $a\in\overline{k}$. Similarly, we define $T_Y(\overline{k})$ and $T_Z(\overline{k})$.

Second, through the embedding $\operatorname{GL}_2(\overline{k})\hookrightarrow\operatorname{PGL}_3(\overline{k}):\,A\mapsto  \left(\begin{array}{cc}
              1 & 0 \\
              0 & A \\
            \end{array}
          \right)$, we can identify $\operatorname{GL}_2(\overline{k})$ with its image in $\operatorname{PGL}_3(\overline{k})$. We denote this image by $\operatorname{GL}_{2,X}(\overline{k})$. In the same way, we define $\operatorname{GL}_{2,Y}(\overline{k})$ and $\operatorname{GL}_{2,Z}(\overline{k})$.

Third, the following subgroups of $\operatorname{PGL}_3(\overline{k})$ are also considered: $\tilde{S_3}:=\langle[X:Z:Y],[Z:X:Y]\rangle,\,G_{03}:=\langle[X:Z:Y],[X:\zeta_3 Y:Z]\rangle,$ and
$G_{05}:=\langle[X:Z:X],\,[X:Y:\zeta_5Z]\rangle.$

\begin{thm}\label{genorm}
Let $\varrho(G)$ be one of the automorphism groups given by Theorem \ref{autoGps}, such that $\widetilde{\mathcal{M}_6^{Pl}}(\varrho(G))$ is not $0$-dimensional. The normalizer $N_{\varrho(G)}(\overline{k})$ of $\varrho(G)$ in $\operatorname{PGL}_3(\overline{k})$ is generated by:
\begin{multicols}{2}
\begin{itemize}
  \item $\operatorname{N}_{\{1\}}(\overline{k})=\operatorname{PGL}_3(\overline{k})$;
  \item $\operatorname{N}_{\varrho(\Z/2\Z)}(\overline{k})=\operatorname{GL}_{2,Z}(\overline{k})$;
  \item $\operatorname{N}_{\varrho(\Z/3\Z)}(\overline{k})=\langle D(\overline{k}), \tilde{S_3}\rangle$;
\item $\operatorname{N}_{\varrho(\Z/4\Z)}(\overline{k})=\operatorname{GL}_{2,Z}(\overline{k})$;
\item $\operatorname{N}_{\varrho(\Z/4\Z)}(\overline{k})=\langle D(\overline{k}), [Z:Y:X]\rangle$;
\item $\operatorname{N}_{\varrho(\Z/5\Z)}(\overline{k})=\operatorname{GL}_{2,Z}(\overline{k})$;
\item $\operatorname{N}_{\varrho(S_3)}(\overline{k})=\langle T_X(\overline{k}), G_{03}\rangle$;
\item $\operatorname{N}_{\varrho(\Z/8\Z)}(\overline{k})=\langle D(\overline{k}), [Z:Y:X]\rangle$;
\item $\operatorname{N}_{\varrho(D_{10})}(\overline{k})=\langle T_X(\overline{k}),\,G_{05}\rangle$ ;
\item $\operatorname{N}_{\varrho(\Z/10\Z)}=D(\overline{k})$.
\end{itemize}
\end{multicols}
\end{thm}

\begin{proof}The Theorem is a straightforward implication from the well-known result that says that two non-singular matrices commute if and only if there is a common basis in which both of them diagonalize or one is a multiple of the identity. As an example, we prove the cases $\varrho(\Z/3\Z)$ and $\varrho(S_3)$ simultaneously, and the remaining situations are proven in the same way: if $\phi\in \operatorname{N}_{\varrho(\Z/3\Z)}(\overline{k})$, then $\phi^{-1}\operatorname{diag}(1,\zeta_3,\zeta_3^2)\phi= \operatorname{diag}(1,\zeta_3,\zeta_3^2),$ or $diag(1,\zeta_3^2,\zeta_3)$. Hence, $\phi$ is diagonal or a permutation of the variables, up to a re-scaling. In particular, $\phi$ is a product of an element of $D(\overline{k})$ and an element of $\tilde{S_3}$, which gives the situation for $\varrho(\Z/3\Z)$. On the other hand $\operatorname{N}_{\varrho(S_3)}(\overline{k})\subseteq \operatorname{N}_{\varrho(\Z/3\Z)}(\overline{k})$. Furthermore, if $\phi\in \operatorname{N}_{\varrho(\Z/3\Z)}(\overline{k})$ such that $\phi^{-1}[X:Z:Y]\phi$ is of order $2$ in $\varrho(S_3)=\langle [X:\zeta_3Y:\zeta_3^2Z],\,[X:Z:Y]\rangle$, then $$\phi\in\{\operatorname{diag}(a,\zeta_3^r,1),[a X:\zeta_3^rZ:Y]\,|\,a\in \overline{k}\,\,and\,\,0\leq r\leq2\}.$$
Rewriting $\operatorname{diag}(a,\zeta_3^r,1)$ as $\operatorname{diag}(a,1,1)\operatorname{diag}(1,\zeta_3,1)^r,$ and $[a X:\zeta_3^rZ:Y]$ as $\operatorname{diag}(a,\zeta_3^r,1)[X:Z:Y]$, gives the conclusion for $\varrho(S_3)$.
\end{proof}

\subsection{Representative families for $\widetilde{\mathcal{M}_6^{Pl}}(\varrho(G))$}
We study the existence of representative families over $k$ for the different strata, i.e families of curves $\verb"C"\rightarrow\verb"S"$ whose points are in natural bijection with the subvarieties.

R. Lercier, et al. in \cite{LeRiRo} explicitly constructed such families for smooth plane quartics in order
to determine unique representatives for the isomorphism classes of smooth plane quartics
over finite fields. We also refer to the second author's thesis \cite[Ch. 2]{Loth} for such parametrization of smooth plane quartic curves over number fields.

\begin{thm}[Representative families]\label{complete} The following table shows representative families over $k$ for each stratum of smooth $\overline{k}$-plane curves of genus $6$, with non-trivial automorphism group of order $\neq5$. For $\widetilde{\mathcal{M}_6^{Pl}}(\varrho(\Z/5\Z))$ a geometrically complete family is shown.
\scriptsize
\begin{center}
\begin{longtable}{|c|c|c|c|}
  \hline
    Case &$G$ & $F_{\rho(G)}(X;Y;Z)$ & \text{Parameters restrictions} \\\hline\hline
 1& $\operatorname{GAP}(150,5)$    & $X^5+Y^5+Z^5$ & - \\
  2&$\operatorname{GAP}(39,1)$     & $X^4Y+Y^4Z+Z^4X$  & -\\
  3&$\operatorname{GAP}(30,1)$     & $X^5+Y^4Z+YZ^4$&-\\
4&$\Z/{20}\Z$   & $X^5+Y^5+XZ^4$   & -\\
  5&$\Z/{16}\Z$     & $X^5+Y^4Z+XZ^4$  & - \\\hline
6&$\Z/{10}\Z$       & $\mathbf{X^5+Y^5+ aXZ^4+X^3Z^2}$&$a\neq 0,1/4$
\\\hline
7&$D_{10}$          & $\mathbf{X^5+a(Y^5+Z^5)+XY^2Z^2+bX^3YZ},$ &$a\neq 0,\,b\neq 1$
\\
&          & $\mathbf{X^5+c(Y^5+Z^5)+X^3YZ}$& $a^3\neq-3^35^{-5}$\\\hline
8&$\Z/8\Z$          & $\mathbf{X^5+Y^4Z+a XZ^4+X^3Z^2}$&$a\neq 0,1/4$ 
\\\hline
 9&$S_3$            &  $\mathbf{a^{3}X^5+Y^4Z+YZ^4+a^{2}X^3YZ+abX^2\big(Z^3+Y^3\big)+cXY^2Z^2},$ & n.s
 \\
  &                 &  $\mathbf{d^{2}X^5+Y^4Z+YZ^4+d X^2\big(Z^3+Y^3\big)+eXY^2Z^2},$ & n.s 
  \\
  &            &  $\mathbf{f^{4}X^5+Y^4Z+YZ^4+fXY^2Z^2}$ & $f\neq0,-\frac{3125}{16}$\\\hline
10&$\Z/5\Z$          & $\mathbf{Z^5+XY(X+Y)(X+aY)(X+bY)}$& $ab(a-1)(b-1)(a-b)\neq0,\,n.b$\\\hline
11&$\Z/4\Z$          & $\mathbf{X^5+X^3Z^2+Y^2Z^3+aX^2Y^2Z+X(bY^4+cZ^4)}$, & $bc\neq 0$, $c\neq-\frac{7}{20}$ 
\\
&         & $\mathbf{X^5+X^2Y^2Z+X(dY^4+eZ^4)+Y^2Z^3}$,& $de\neq 0$ \\
&         & $\mathbf{X^5+f(Y^2Z^3+X(Y^4+Z^4))}$ & $f^3\neq-(\frac{3}{4})^3,0$ \\\hline

12 &$\Z/4\Z$          & $\mathbf{X^5+c(X^3Y^2+aX^2Y^3+bcXY^4+c^2Y^5+Z^4Y)}$  &   n.s,\,n.b
\\
   &                  & $\mathbf{X^5+s(XY^4+Y^5+Z^4Y)}$  & \\
   &                  & $\mathbf{X^5+e(X^2Y^3+fXY^4+eY^5+Z^4Y)},$  &  \\
   &                  & $\mathbf{X^5+g(X^2Y^3+XY^4+Z^4Y)}$  &   \\\hline
13&$\Z/3\Z$         & $\mathbf{a^{3}X^5+Y^4Z+YZ^4+a^{2}X^3YZ+aX^2\big(bY^3+cZ^3\big)+dXY^2Z^2},$  &
\\
   &         & $\mathbf{e^2X^5+Y^4Z+YZ^4+X^2\big(eY^3+fZ^3\big)+gXY^2Z^2,}$  & n.s,\,$e\neq f$ \\
   &         & $\mathbf{h^2X^5+Y^4Z+YZ^4+hX^2Z^3+sXY^2Z^2,}$  & n.s \\
   &         & $\mathbf{t^2X^5+Y^4Z+YZ^4+tX^2Z^3}$  & $t\neq0,\,\frac{3125}{1024}$ \\\hline
  14&$\Z/2\Z$         & $\mathbf{Z^4Y+Z^2(X^3+XY^2+aY^3)+L_{5,Z}}$   &  n.s,\,n.b \\\hline
 \caption{Complete families over $k$}\label{table:Complete}
\end{longtable}
\end{center}
\normalsize
\end{thm}

\vspace{-1.2cm}

The families that are modified respect to the ones in Table \ref{table:FullAuto.} are highlighted. The automorphism groups remain the same one than in Table \ref{table:FullAuto.}.

The parameters restrictions come from avoiding singular equations and larger automorphism groups. We use the abbreviations ``n.s" and ``n.b" for non-singularity and no bigger automorphism group, when it is tedious to write down the restrictions.

\begin{proof}(Cases with $G\ncong\mathbb{Z}/5\mathbb{Z}$)
Clearly, the zero dimensional strata families are representative over $k$, since each represents a single point in the (coarse) moduli space $\mathcal{M}_6$. For the rest of cases, except for the case with $G\simeq\mathbb{Z}/5\mathbb{Z}$, we will use the same techniques used in \cite{LeRiRo} and \cite{Loth}. We give a detailed example with the case $G\simeq\operatorname{D}_{10}$.

We start with the more general family than in Theorem \ref{autoGps}: $aX^5+b(Y^5+Z^5)+cX^3YZ+dXY^2Z^2=0$. We know $a,b\neq0$ to avoid getting the singular points $(1:0:0)$ and $(0:1:0)$. So, after rescaling and renaming variables, we obtain the geometrically complete family: $X^5+Y^5+Z^5+aX^3YZ+bXY^2Z^2=0$.

If $b=0$, then $a\neq0$ to avoid having a bigger automorphism group, so again, after rescaling and renaming variables, we can work with the family: $X^5+c(Y^5+Z^5)+X^3YZ=0$. Now, it is easy to check that all the matrices in $\operatorname{N}_{\rho(\operatorname{D}_{10})}$  (see Theorem \ref{genorm}) carrying equations in this family into equations again in the family, leave the equation invariant. In other words, a curve with parameter $a$ in this family is only isomorphic to itself, which implies that the component in the family is representative over $k$.

If $b\neq0$, we work with the family $X^5+a(Y^5+Z^5)+bX^3YZ+XY^2Z^2=0$. Again any matrix in   $\operatorname{N}_{\rho(\operatorname{D}_{10})}$ leaving the family invariant, fixes each equation.

So, the family given by these two components is representative over $k$ for the stratum with $\operatorname{D}_{10}$.
\end{proof}

Before proving the last part if the Theorem, we need some previous results.

\begin{lem}\label{whygeom}
The family $\mathcal{C}_{(a,b)}:\,Z^5+XY(X+Y)(X+aY)(X+bY)=0$ is a geometrically complete family over $k$ for the stratum of smooth $\overline{k}$-plane curves of genus $6$, with automorphism group isomorphic to $\Z/5\Z$. In particular, the associated scheme has dimension $2$.
\end{lem}
\begin{proof}
The family $Z^5+L_{5,Z}=0$ is a geometrically complete family over $k$ for the stratum, by Theorem \ref{autoGps}. Moreover, $L_{5,Z}$ should factored  in $\overline{k}[X,Y]$ into pairwise distinct linear factors, otherwise, it will be singular. Now, up to $\overline{k}$-isomorphism, we change $X$ and $Y$, separately, to make one of the factors equals to $X=0$ and another to $Y=0$. Second, re-scale $X,\,Y$ and $Z$ simultaneously to get the factor $(X+Y)$ in the factorization of $L_{5,Z}$. Now, we can write the family as $\mathcal{C}_{(a,b)}:\,Z^5+XY(X+Y)(X+aY)(X+bY)=0$. This family is geometrically complete over $k$ for $\widetilde{\mathcal{M}_6^{Pl}}(\varrho(\Z/5\Z))$, and finite (we justify this next), so the dimension is $2$.

Isomorphisms from the curve $\mathcal{C}_{(a,b)}$ to another curve in this family come from transformations
$$
\begin{pmatrix}\alpha&\beta\\ \gamma & \delta\end{pmatrix}:\,t\mapsto\frac{\alpha t+\beta}{\gamma t+ \delta},
$$
sending the set $\{0,1,\infty,a,b\}$ to a set $\{0,1,\infty,c,d\}$. The set $\mathcal{T}$ of such transformations is a group and it is isomorphic to $S_5$. Moreover, it is generated by
$$
\tau_1(a,b)=(a,\frac{a(b-1)}{b-a}),\,\,\,\,\,\tau_2(a,b)=(\frac{1}{b},\frac{a}{b}),\,\,\,\,\,\tau_3(a,b)=(b,a).
$$
The latest does not properly define a transformation of the curve in the family since switching the parameters $a,b$ does not change the equation. The first two satisfy the relations $\tau_1^2=\tau_{2}^{3}=(\tau_1\tau_2)^5=1$ generating a group isomorphic to $\operatorname{A}_5$.

The family defined in this way $\verb"C"\rightarrow\verb"S"$  is finite and the fibers of $f_{\verb"C"}: \verb"S"\rightarrow S$ have cardinality \footnote{Another way of checking the cardinality is starting with a generic curve $Z^5+\prod_{i=0}^{5}(X+\alpha_i Y)$ and counting the $(5\cdot4\cdot3)\cdot2$ ways of choosing the $3$ roots going to $\infty,0,1$ and getting the parameters $a$ and $b$} $120$.

\end{proof}

The family $\mathcal{C}_{(a,b)}$ is defined over $\overline{k}(a,b)$. We are ideally looking for a family (with two parameters since we know the dimension is two) defined over $L=\overline{k}(a,b)^{\mathcal{T}}$. Hence, we look for the Galois descent from $\overline{k}(a,b)$ to $L$.  This the idea behind the ad-hoc method used in \cite{LeRiRo}, see the proof of Theorem $3.3$.

The following asserts that the analogue of L\"{u}roth's theorem \cite[Chapter IV, 2.5.5]{hart} holds in dimension $2$.
\begin{thm}\label{Luroth}\cite[Chapter V, Theorem 6.2 and Remark 6.2.1]{hart}
Let $L/K$ be a subfield extension of a purely transcendental extension $K(a,b)/K$, where $K$ is an algebraically closed field. If $K(a,b)$ is a finite separable extension of $L$, then $L$ is also a pure transcendental extension of $K$.
\end{thm}

We can now prove:

\textbf{Claim.} There exists a representative family over $k$, for the stratum $\widetilde{\mathcal{M}_6^{Pl}}(\varrho(\Z/5\Z))$ of smooth $\overline{k}$-plane curves of genus $6$, and the proof of Theorem \ref{complete} is finished.

\begin{proof}
Consider the family $\mathcal{C}_{(a,b)}$, which is geometrically complete over $k$ (see Lemma \ref{whygeom}). It is defined over $\overline{k}(a,b)$, or to be more precise, it is already defined over the subextension $\overline{k}(c,d)$ where $c=a+b$ and $d=ab$. We want to descend it to the invariant subfield $L$ under the action of the symmetric group $\operatorname{S}_5$ on $\overline{k}(a,b)$, or equivalently of $\operatorname{A}_5$ on $\overline{k}(c,d)$. In this way, we will get a representative family over $k$: if such a family exists, it is purely transcendental of dimension $2$ by Theorem \ref{Luroth}, and hence given by $2$ parameters.

On the other hand, Weil's cocycle criterion \cite{We} states that if there exists a family of isomorphisms
$$\{\phi_{\sigma}:\,^{\sigma}\mathcal{C}_{(a,b)}\rightarrow\mathcal{C}_{(a,b)}\}_{\sigma\in Gal(\overline{L}/L)},$$
satisfying the cocycle condition $\phi_{\sigma_1\sigma_2}=\phi_{\sigma_1}\,^{\sigma_1}\phi_{\sigma_2}$, then there exist a descend of the curve  $\mathcal{C}_{(a,b)}$ over $L$.
That is an isomorphism $\phi:\,\mathcal{C}\rightarrow\mathcal{C}_{(a,b)}$ of smooth curves, satisfying $\phi\,\circ\,^{\sigma}\phi^{-1}=\phi_\sigma$ and where $\mathcal{C}$ is defined over $L=\overline{k}(\alpha,\beta)^{\mathcal{T}}$.

A priori, the curve $\mathcal{C}$ does not need to be a plane curve, even if the isomorphisms $\phi_\sigma$ are (projective) matrices and $\mathcal{C}_{(a,b)}$ are plane curves, see Proposition $2.19$ in \cite{BaBaEl1}. But in this case we can conclude that the isomorphism $\phi$ is also defined by matrix since all matrices $\phi_\sigma$ that we will take can be seeing not only as matrices in $\operatorname{PGL}_3(\overline{k}(a,b))$, but as matrices in $\operatorname{GL}_2(\overline{k}(a,b))$, since all of them will have the shape
$$
\begin{pmatrix}
* & * & 0\\
* & * & 0\\
0 & 0 & 1
\end{pmatrix},
$$
and then Hilbert's Theorem 90 implies the existence of a matrix $\phi$ satisfying $\phi\,\circ\,^{\sigma}\phi^{-1}=\phi_\sigma$.
Another proof of this fact is \cite[Theorem 2.6]{BaBaEl1}, that says that a $\overline{K}$-plane curve defined over $K$ has a $K$-plane model when the degree is coprime to $3$, so even if $\mathcal{C}$ is not given by $\phi$ by a plane model over $L$, we can find one.

The Galois group $\operatorname{Gal}(\overline{k}(a,b)/L)=\operatorname{S}_5$ and generated by $\tau_1$, $\tau_2$ and $\tau_3$. We define $\phi_{\tau_1}=[\lambda(-X+Y):\lambda(\frac{-1}{a}X+Y):Z]$ where $\lambda=\sqrt[5]{(a-1)^{-2}(a-b)^{-1}a^{3}}$, $\phi_{\tau_2}=[\sqrt[5]{b}X:\frac{\sqrt[5]{b}}{b}Y:Z]$ and $\phi_{\tau_3}=[X:Y:Z]$ and the family $\{\phi_\sigma\}_{\sigma\in\operatorname{Gal}(\overline{k}(a,b)/L)}$ by extending with the cocycle condition. This gives a well-defined family of isomorphisms satisfying the Weil cocycle condition. Hence, the mentioned descend exists.
\end{proof}

Indeed, we can construct a explicit matrix $\phi$ by taking a sufficiently general matrix $M$, and making Hilbert's Theorem $90$ explicit:
$$\phi=\sum_{\tau\in Gal(\overline{k}(a,b)/L)}\,\phi_{\tau}\,^{\tau}M,$$
gives an isomorphism satisfying\footnote{It is a straightforward computation to check that $\phi_{\sigma}\,^{\sigma}\phi=\phi$ and the meaning of a sufficiently general matrix $M$ is that the matrix $\phi$ constructed in that way is invertible.} $\phi_{\sigma}=\phi\,\circ\,^{\sigma}\phi^{-1}$.

\subsection{Detecting representatives in the list}
Given a non-singular degree $5$ plane curve with known non-trivial automorphism group, we can find its representative in the classification in Theorem \ref{complete} by using the remarks at the end of chapter $2$ in the second author's Thesis \cite{Loth}.

\section{Stratification of $\mathcal{M}_6^{Pl}$ by automorphism groups}
The diagram in next page shows how looks like the stratification by automorphism groups of non-singular plane quintic curves. The dimensions are computed by looking at the dimension of the schemes $\verb"S"$ in the representative families in Theorem \ref{complete}.

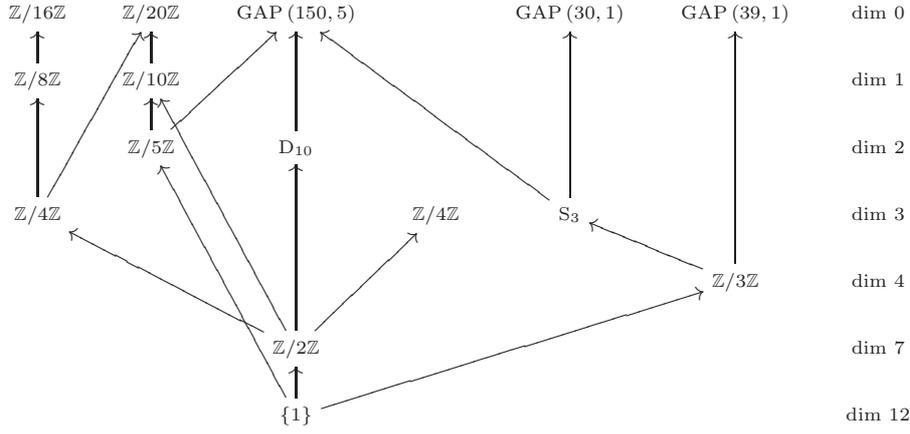
\begin{figure}[!ht]
\scriptsize
$$
\xymatrix@C-=0.5cm@R-=0.4cm{
\mathbb{Z}/16\mathbb{Z}  &  \mathbb{Z}/20\mathbb{Z}  &  \operatorname{\operatorname{GAP}}\left(150,5\right) &  & \operatorname{\operatorname{GAP}}\left(30,1\right) & \operatorname{\operatorname{GAP}}\left(39,1\right) & \operatorname{dim}\, 0  \\
\mathbb{Z}/8\mathbb{Z}  \ar[u] &  \mathbb{Z}/10\mathbb{Z}  \ar[u]    &  &   & & &  \operatorname{dim}\,1  \\
 &  \mathbb{Z}/5\mathbb{Z}  \ar[u]\ar[ruu] & \text{D}_{10} \ar[uu]& & & &  \operatorname{dim}\,2  \\
  \mathbb{Z}/4\mathbb{Z}\ar[uu]\ar[ruuu]  &     & &\mathbb{Z}/4\mathbb{Z}& \text{S}_3\ar[uuu]\ar[lluuu]& &   \operatorname{dim}\,3  \\
  &    && && \mathbb{Z}/3\mathbb{Z} \ar[lu] \ar[uuuu]&   \operatorname{dim}\,4  \\
    &    &  \mathbb{Z}/2\mathbb{Z} \ar[lluu]\ar[luuuu]\ar[uuu]\ar[ruu] & & & &   \operatorname{dim}\,7 \\
        &    &  \{1\} \ar[luuuu]\ar[u]\ar[rrruu] &&  & &   \operatorname{dim}\,12
}
$$
\caption{Stratification by automorphisms group}
\label{diaguay}
\end{figure}
\normalsize

We can see in Figure \ref{diaguay} a phenomenon that does not happen for degree $4$. We define a final stratum by automorphism groups to be a stratum that does not properly contain any other stratum. There is a final stratum in $\mathcal{M}_6^{Pl}$ of dimension greater than zero. This may sound odd since we could expect that by adding conditions in the parameters we would get bigger automorphisms groups. However, we will see that this is a normal situation for degree higher than $4$.

\subsection{A canonical justification}
The family for the stratum with $\rho(G)=\langle[X:iY:-Z]\rangle$ has generic component $\mathcal{C}_{a,b,c}:\,X^5+X^3Z^2+Y^2Z^3+aX^2Y^2Z+bXY^4+cXZ^4=0$. This stratum has dimension $3$ and it is a final stratum: no restriction of the parameters give a larger automorphism group. Let us work for a while with the renormalized family $\mathcal{S}_{A,B,C}:\,X^5+AX^3Z^2+BX^2Y^2Z+CXY^4+XZ^4+Y^2Z^3=0$. Now, it is easy to see that making $A=B=C=0$ we get a larger automorphism group. For instance, we get the new automorphism $[X:\zeta_8 Y:i Z]$. However, the plane curve defined by this equation is singular.

We will see an explanation of this final stratum not having dimension zero. We will regard the family $\mathcal{S}_{A,B,C}$ in $\mathcal{M}_{6}^{Pl}$ inside a family in $\mathcal{M}_6$ that is not final. When we add restrictions here, we get extra symmetries and the curve is not plane anymore.

Let us start by computing the family $\mathcal{K}_{A,B,C}$ of canonical models of $\mathcal{S_{A,B,C}}$. We define the functions $x=X/Z$ and $y=Y/Z$.
\begin{eqnarray*}
\text{div}(x)&=&2(0:0:1)+2(0:1:0)-\sum_{s=1}^{4} (i^{s}\sqrt[4]{-C}:1:0)=2P+2Q-\sum_{i=1}^{4}R_i\\
\text{div}(y)&=&(0:0:1)-(0:1:0)-\sum_{s=1}^{4} (i^{s}\sqrt[4]{-C}:1:0)+(1:0:\pm\sqrt{\tfrac{-A\pm\sqrt{A^2-4}}{2}})=P-Q-\sum_{i=1}^{4}R_i+\sum_{i=1}^{4}T_i.
\end{eqnarray*}

In order to compute $\text{div}(dx)$, we work with the affine form $F(x,y,1)=x^5+Ax^3+Bx^2y^2+Cxy^4+x+y^2$. The differential $dx$ is an uniformizer for all points except for $P$ and the $T_i$'s because the tangent space to the curve at these points have equation $x-\alpha$ for some $\alpha\in\overline{k}(A,B,C)$ (we have used $d(x-\alpha) = dx)$. Then, for those points, we have to work with the expression
$$
dx=-\frac{y(2Bx^2+4Cxy^2+2)}{5x^4+3Ax^2+2Bxy^2+Cy^4+1}dy.
$$
We finally get
$$
\text{div}(dx)=P+Q+\sum_{i=1}^{4}R_i+\sum_{i=1}^{4}T_i,
$$
and a basis of regular differentials is given by
\begin{eqnarray*}
\omega_0&=&\frac{dx}{y},\,\omega_1=\frac{xdx}{y},\,\omega_2=dx,\\
\omega_3&=&\frac{x^2dx}{y},\,\omega_4=xdx,\,\omega_5=ydx.
\end{eqnarray*}

\begin{prop}\label{generators} The ideal of the canonical model of $\mathcal{S}_{A,B,C}$ in $\mathbb{P}^5[\omega_0,\omega_1,\omega_2,\omega_3,\omega_4,\omega_5]$ is generated by the polynomials
$$\omega_0\omega_3=\omega_{1}^{2},\,\omega_0\omega_4=\omega_{1}\omega_2,\,\omega_0\omega_5=\omega_{2}^{2},\,\omega_4^2=\omega_3\omega_5,\,{\omega_1\omega_5=\omega_2\omega_4,\,\omega_1\omega_4=\omega_2\omega_3},$$
$$\omega_1\omega_{3}^2+A\omega_{1}^{3}+B\omega_0\omega_{4}^{2}+C\omega_2\omega_4\omega_5+\omega_{0}^2\omega_1+\omega_{0}^2\omega_5=0,$$
$$\omega_{3}^3+A\omega_0\omega_{3}^{2}+B\omega_1\omega_{3}\omega_{5}+C\omega_3\omega_{5}^{2}+\omega_{0}^2\omega_3+\omega_{0}\omega_1\omega_5=0,$$
$$\omega_4\omega_{3}^2+A\omega_{0}\omega_{3}\omega_{4}+B\omega_2\omega_{3}\omega_5+C\omega_4\omega_{5}^{2}+\omega_{0}^2\omega_4+\omega_{0}\omega_2\omega_5=0.$$
We denote it by $\mathcal{K}_{A,B,C}$.
\end{prop}

\begin{proof}
If $\omega_0\neq 0$, then the des-homogenization of this ideal with respect to $\omega_0$ gives $$\omega_1^5+A\omega_1^3+B\omega_1^2\omega_2^2+C\omega_1\omega_2^4+\omega_1^4+\omega_2^2,$$ and we recover the affine curve $\mathcal{S}_{A,B,C}$ for $Z=1$. If $\omega_0=0$, then $\omega_1=\omega_2=0$, $\omega_3\omega_4=\omega_5^2$, $\omega_3(\omega_3^2+\omega_5^2)=0$ and we recover the points at infinity for $\mathcal{S}_{A,B,C}$: $Q$, $R_i$'s.

To check that it is non-singular, we need to see if the rank of the matrix of partial derivatives of the previous generating functions has rank equal to $\text{dim}(\mathbb{P}^5)-\text{dim}(\mathcal{K}_{A,B,C})=4$  at every point\footnote{here we mean de dimension of $\mathcal{K}_{A,B,C}$ as a variety over $k(A,B,C)$, that is $1$ since it is a curve}, that is, the tangent space has codimension $4$. If $\omega\neq0$, the partial derivatives of the first three equations plus the equation in the second line produce $4$ linearly independent vectors in the tangent space. If $\omega_0=0$, then $\omega_5$ is non-zero, and, the $3$rd, $4$th, and the $6$th equations plus the equation in the last line provide $4$ linearly independent vectors.
Moreover, this is independent of the choice of the parameters $A,B,C$.
\end{proof}

\begin{cor}
If we specialize the parameters to $A=B=C=0$, then we get a smooth curve of genus $6$, whose full automorphism group has order multiple of $8$ and contains
$$\operatorname{diag}(\zeta_{8}^{7}, \zeta_{8}^{5}, \zeta_{8}^{6}, \zeta_{8}^{3}, \zeta_{8}^{4}, \zeta_{8}^{5}).
$$
This curve does not admit a non-singular plane model over $\overline{k}$. So the $11$th stratum of plane curves of genus $6$ in Theorem \ref{complete} that is final as a plane stratum, is indeed living inside a stratum of smooth curves of genus $6$ which is not final.
\end{cor}

\begin{proof} We only need to check that the curve $\mathcal{K}_{0,0,0}$ is not isomorphic to any one in the family $8$th in Theorem \ref{autoGps} with automorphism group $\mathbb{Z}/8\mathbb{Z}$. In order to check that we just look at the automorphism $[X:\zeta_8 Y:-Z]$ of order $8$ in this family acting in its canonical model. We mimic the previous computations and we get the matrix $\operatorname{diag}(\zeta_{8}^{5}, \zeta_{8}^{6}, \zeta_{8}, \zeta_{8}^{7}, \zeta_{8}^{2}, \zeta_{8}^{5})$. the group generated by this matrix is clearly non conjugated to the group generated by the one in the statement of the corollary. So, the curve  $\mathcal{K}_{0,0,0}$ does not have a smooth plane model.
\end{proof}

We reinterpret the existence of these special kind of families of plane curves in terms of $g^r_d$ linear series as follows: suppose that $\mathcal{C}$ is such a family describing a stratum of plane curves, and let $D$ be a divisor in $\operatorname{Div}(\mathcal{C})$ that defines a $g^2_{d}$ linear series for $\mathcal{C}$. In particular, $D$ is of degree $d$, and the vector space $\mathcal{L}(D)$ has dimension $2$. For some specializations of the parameters in the canonical family $\mathcal{K}$, given by the canonical embedding $\Phi:\,\mathcal{C}\hookrightarrow\mathbb{P}^{g-1}(\overline{k})$, one gets more automorphisms. Hence more symmetries in the defining equations, which in turns produce more meromorphic functions with poles bounded above by $\mathcal{D}=\Phi(D)$. Therefore, $dim(\mathcal{L}(\mathcal{D}))>2$, and we do not get a smooth $\overline{k}$-plane model anymore.

In our example the divisor $D$ generating the $g_{2}^{d}$-linear system is $D=Q+\sum_{i=1}^{4}R_i$, and $\mathcal{L}(\mathcal{D})$ is generated by $1,\frac{\omega_1}{\omega_0},\frac{\omega_2}{\omega_0}$. For the special choice of the parameters $A=B=C=0$, we get $D=5Q$ and $\mathcal{L}(\mathcal{D})$ contains the linearly independent functions  $1,\frac{\omega_1}{\omega_0},\frac{\omega_2}{\omega_0},\frac{\omega_3}{\omega_0}$, so the (projective) dimension of $\mathcal{L}(\mathcal{D})$ is greater than $2$ and it does not define a $g_{2}^{d}$ linear system anymore. This why, for this choice of the parameters we do not get a smooth plane model anymore, and it is due to the extra symmetries of the curve.

\subsection{Non-zero dimensional final strata for higher degree}
In this subsection we show examples of non-zero dimensional final strata in $\mathcal{M}_{g}^{Pl}$ for infinitely many $g$'s. 

To prove that they are final strata, that is, that there cannot be more automorphisms, we use similar techniques to the ones used in \cite{BaBa1, BaBa3}. In particular, we need the next Theorem, which follows when one combines \cite[\S 1-10]{Mit} or \cite[Theorem 4.8]{dolgachev2009finite}, and the proof of \cite[Theorem 2.1]{Ha}):

\begin{thm}[Mitchell, Harui]\label{Harui,Mitchell}
	Let $G$ be a subgroup of automorphisms of a smooth $\overline{k}$-plane curve $C$ of degree $d\geq4$, where $k$ has characteristic $p=0$. Then one of the following situations holds:
	\begin{enumerate}
		\item $G$ fixes a line and a point off this line.
		\item $G$ fixes a triangle and neither line nor a point is leaved invariant. In this case, $(C,G)$ is a descendant of the the Fermat curve $F_d:\,X^d+Y^d+Z^d=0$ or the Klein curve $K_d:\,XY^{d-1}+YZ^{d-1}+ZX^{d-1}=0$.
		\item $G$ is conjugate to a finite primitive subgroup of $\operatorname{PGL}_3(\overline{k})$ namely, the Klein group
		$\operatorname{PSL}(2,7)$, the icosahedral group $A_5$, the alternating group
		$\operatorname{A}_6$, the Hessian group $\operatorname{Hess}_{*}$ with $*\in\{36,72,216\}$.
		\end{enumerate}
\end{thm}

\begin{rem}\label{postive}
We follow the discussion in \cite[\S6]{BaBa1} to deduce that Theorem \ref{Harui,Mitchell} holds when the characteristic $p>2g+1$, where $g$ denotes the geometric genus of the curve $C$.
\end{rem}
From now on we assume that $k$ has characteristic $p=0$ or $p>2g+1$.
 
First, we prove:
\begin{thm}\label{odddegreefinal}
Let $\mathcal{C}$ be the family of smooth $\overline{k}$-plane curves of an odd degree $d\geq7$, defined in the following way:
\begin{itemize}
  \item if $d\equiv1\text{ mod }4$, then
    $$\mathcal{C}:\,X^d+X^{d-2}Y^2+X^{(d-1)/2}YZ^{(d-1)/2}+aXY^{d-1}+bXZ^{d-1}+cY^{(d+1)/2}Z^{(d-1)/2}=0$$

  \item if $d\equiv3\text{ mod }4$, then
  $$\mathcal{C}:\,X^d+X^{d-2}Y^2+X^{(d-1)/2}YZ^{(d-1)/2}+aXY^{d-1}+bXZ^{d-1}+cX^2Y^{(d-1)/2}Z^{(d-1)/2}=0$$
\end{itemize}
Then $\operatorname{Aut}(\mathcal{C})=\langle[X:-Y:\zeta_{d-1}Z]\rangle$, hence it is isomorphic to $\Z/(d-1)\Z$.
\end{thm}

\begin{proof}
We first note that $\eta:\,(X:Y:Z)\mapsto (X:-Y:\zeta_{d-1}Z)$ defines an automorphism of $\mathcal{C}$ of order $d-1$, and also $ab\neq0$ by non-singularity. we will see that $\operatorname{Aut}(\mathcal{C})$ is not conjugate to any of the primitive groups in $\operatorname{PGL}_3(\overline{k})$ mentioned in Theorem \ref{Harui,Mitchell}. In particular, it should fix a point, a line, or a triangle.

We just handle the situation for the Hessian groups $\operatorname{Hess}_{*}$, since $\operatorname{A}_6$ (so does $\operatorname{A}_5$) has no elements of order $d-1\geq6$, and the Klein group $\operatorname{PSL}(2,7)$ (the only simple group of order $168$) has no elements of order $6$ or $\geq8$ inside. Moreover, any element of $\operatorname{Hess}_{*}$
has order $1,\,2,\,3,\,4,$ or $6$. Then $\operatorname{Hess}_{*}$  does not appear as the full automorphism group of $\mathcal{C}$, except possibly for $d=7$ and $*=216$. But, following the same discussion that we did for \cite[Proposition 12]{BaBa3}, one easily deduces that there exists no smooth $\overline{k}$-plane curve of degree $7$, whose automorphism group is conjugate to $\operatorname{Hess}_{216}$.

It is well known that $|Aut(F_d)|=6d^2$ and $|Aut(K_d)|=3(d^2-3d+3)$, see, for instance, \cite[Propositions $3.3,\,3.5$]{Ha}. Hence, $\mathcal{C}$ can not be a descendant of the Klein curve $K_d$ or the Fermat curve $F_d$, since, $d-1$ does not divide $3(d^2-3d+3),$ or $6d^2$ (except possibly when $d=7$). On the other hand, the automorphisms of $F_7:\,X^7+Y^7+Z^7=0$ have the shape $[X:\zeta_{7}^aY:\zeta_{7}^bZ],\,[\zeta_{7}^bZ:\zeta_{7}^aY:X],\,[X:\zeta_{7}^bZ:\zeta_{7}^aY],\,[\zeta_{7}^aY:X:\zeta_{7}^bZ],
\,[\zeta_{7}^aY:\zeta_{7}^bZ:X],$ or $[\zeta_{7}^bZ:X:\zeta_{7}^aY].$ Non of these automorphisms has has order $6$, so $\mathcal{C}$ with $d=7$ is not a descendant of $F_7$, as well. Therefore, $\operatorname{Aut}(\mathcal{C})\hookrightarrow PGL_{3}(\overline{k})$ must fix a line and a point off this line.

In particular, the fixed line is one of the reference lines $B=0$ with $B\in\{X,Y,Z\}$, and the point is one of the reference points $(1:0:0),\,(0:1:0),$ or $(0:0:1)$, because $\eta\in\operatorname{Aut}(\mathcal{C})$ does. Consequently, all automorphisms of $\mathcal{C}$ are all of the shapes either $[X:vY+wZ:sY+tZ],\,[vX+wZ:Y:sX+tZ],$ or $[vX+wY:sX+tY:Z]$. In any case, we always get $s=w=0$ through the term $X^{d-2}Y^2$, and all automorphisms of $\mathcal{C}$ are diagonal, say $\operatorname{diag}(1,\lambda,\mu).$ We then obtain $\lambda^2=\lambda\mu^{(d-1)/2}=\lambda^{d-1}=\mu^{d-1}=1$, that is, $\lambda=1,\,\mu=\zeta_{d-1}^{2r}$ or $\lambda=-1,\,\mu=\zeta_{d-1}^{2r+1}$ for $r\in\{0,1,...,d-2\}$. Consequently, $|\operatorname{Aut}(\mathcal{C})|=d-1$.	
\end{proof}

\begin{cor}\label{dimffodd}
	The family $\mathcal{C}$ given by Theorem \ref{odddegreefinal} describes a final stratum of dimension $3$.
\end{cor}

Second, for even degrees $d\geq10$, we construct the family:
\begin{thm}\label{evendegreefinal}
For any $d=2m\geq10$ with $m$ odd, let $\mathcal{C'}$ be the family of smooth $\overline{k}$-plane curves defined by  $$X^{d}+Y^{d}+Z^{d}+X^{m}Z^m+X^{d-2}Y^2+a_2X^{d-4}Y^{4}+a_3X^{d-6}Y^{6}+...+a_{m-1}X^{2}Y^{d-2}+a_mX^{m-2}Y^2Z^{m}=0.$$
Then $\operatorname{Aut}(\mathcal{C'})=\langle[X:-Y:\zeta_mZ]\rangle$, in particular, it is cyclic of order $d$.

\end{thm}

\begin{proof}
Similarly, $\operatorname{Aut}(\mathcal{C'})$ is not conjugate to any of the finite primitive subgroups of $\operatorname{PGL}_{3}(\overline{k})$, and also $\mathcal{C}'$ is not a descendant of the Klein curve $K_d$. So $\operatorname{Aut}(\mathcal{C'})$ fixes a line and a point off this line, or $\mathcal{C'}$ is a descendant of the Fermat curve $F_d$. In the first situation, we deduce that it is cyclic of order $d$, by using the same discussion of Theorem \ref{odddegreefinal} through the monomials $X^mZ^m$ and $X^{d-2}Y^2$. On the other hand, any automorphism of $F_d$ of order $d$, which is not a homology type\footnote{i.e., it fixes only three points in the plane} is conjugate, inside $\operatorname{Aut}(F_d)$, to either $[X:-Y:\zeta_mZ]$, or $[X:\zeta_mZ:-Y]$. But also both automorphisms are in different conjugacy classes in $\operatorname{PGL}_{3}(\overline{k})$, so, if $\mathcal{C}$ is a descendant of the Fermat curve, through a transformation $A\in\operatorname{PGL}_3(\overline{k})$, then we may assume that $A^{-1}[X:-Y:\zeta_mZ]A=[X:-Y:\zeta_mZ]$. In particular, the matrix $A$ must be a diagonal one, say $\operatorname{diag}(1,\lambda,\mu)$ such that $\lambda^{d}=\mu^{d}=\mu^{m}=\lambda^2=1$. Therefore, $\lambda=\pm1$ and $\mu=\zeta_m^a$ for some integer $a$, and thus $A\in\operatorname{Aut}(\mathcal{C})$. Moreover, $\operatorname{Aut}(\mathcal{C})\cap\operatorname{Aut}(F_d)=\langle[X:-Y:\zeta_mZ]\rangle$. This completes the proof.

\end{proof}

\begin{cor}\label{corevendegree}
	The family $\mathcal{C'}$ in Theorem \ref{evendegreefinal} describes a final stratum of dimension at least $m-3$.
\end{cor}

\begin{rem}\label{importantffeven}
Theorem \ref{evendegreefinal} and Corollary \ref{corevendegree} are still true when $m$ is even, but for simplicity we need to impose $m\geq8$ in order to exclude the primitive groups in $\operatorname{PGL}_{3}(\overline{k})$.
\end{rem}


\begin{thebibliography}{10}
\bibitem{Es} E. Badr, \emph{On smooth plane curves}, PhD in
progress at UAB.
\bibitem{BaBa1} E. Badr, F. Bars, \emph{On the locus of smooth plane curves with a fixed automorphism
group}, Mediterr. J. Math. \textbf{13} (2016), 3605-3627. doi:\,10.1007/s00009-016-0705-9.
\bibitem{BaBa3}E. Badr,  F. Bars, \emph{Automorphism groups of non-singular plane curves of degree 5 },
Commun. Algebra \textbf{44} (2016), 327-4340.
doi:\,10.1080/00927872.2015.1087547.
\bibitem{BaBaEl1}E. Badr, F. Bars, E. Lorenzo Garc\'ia, \emph{On twists of smooth plane curves}, arXiv:1603.08711v1.
\bibitem{dolgachev2009finite} Dolgachev I. and Iskovskikh V.; \emph{Finite subgroups of the plane Cremona group}, Algebra,
Arithmetic, and Geometry, Progress in Mathematics Volume \textbf{269} (2009), 443-548
\bibitem{GAP}  The GAP Group, GAP-Groups, Algorithms,  and Programming, Version 4.5.7 (2012), http://www.GAP-system.org.
\bibitem{hart} R. Hartshorne, \emph{AAlgebraic geometry}, Springer-Verlag, New York (1977).
\bibitem{Ha}T. Harui, \emph{Automorphism groups of plane
curves}, arXiv: 1306.5842v2.
\bibitem{Hug} B. Huggins, \emph{Fields of moduli and fields of definition of
curves}. PhD thesis, Berkeley (2005), arxiv.org/abs/math/0610247v1.
\bibitem{LeRiRo} R. Lercier, C. Ritzenthaler, F. Rovetta, J. Sijsling, \emph{Parametrizing the moduli space of curves and applications to smooth plane quartics over finite fields},
(LMS Journal of Computation and Mathematics, Volume 17, Special
Issue A (ANTS XI), LMS, London, pp. 128--147 (2014).
\bibitem{Loth} E. Lorenzo Garc\'ia, \emph{Arithmetic properties of non-hyperelliptic genus 3
curves}, PhD dissertation, Universitat Polit\`ecnica de Catalunya
(2015), Barcelona.
\bibitem{Lo2} E. Lorenzo Garc\'ia, \emph{Twists of non-hyperelliptic curves of genus $3$}, arXiv:1604.02410.
\bibitem {Mit} H. Mitchell , \emph{Determination of the ordinary and modular ternary linear groups}, Trans.
Amer. Math. Soc. \textbf{12}, no. 2 (1911), 207-242.
\bibitem{We} A. Weil, \emph{The field of definition of a variety}, American J. of Math. vol. \textbf{78}, n17 (1956),
509-524.
\end{thebibliography}
\end{document}